\documentclass[11pt]{article}
\usepackage[T1]{fontenc}
\usepackage[utf8]{inputenc}
\usepackage{lmodern,microtype}
\usepackage{amsmath}
\usepackage{amsthm}
\usepackage[initials]{amsrefs} 
\usepackage{amssymb}
\usepackage{amsfonts,mathtools}
\usepackage{setspace}
\usepackage{fullpage}
\usepackage{todonotes}
\usepackage{enumitem}
\usepackage{bbold} 
\usepackage{hyperref}
\hypersetup{colorlinks=true,
   citecolor=blue,
   filecolor=blue,
   linkcolor=blue,
   urlcolor=blue
}
\usepackage{cleveref}

\newtheorem{lemma}{Lemma}[section]
\newtheorem{theorem}{Theorem}[section]
\newtheorem{corollary}{Corollary}[section]
\newtheorem{definition}{Definition}[section]
\newtheorem{problem}{Problem}[section]
\newtheorem{proposition}{Proposition}[section]

\newcounter{cas}

\newtheoremstyle{assert}
  {.5\baselineskip±.2\baselineskip}   
  {.5\baselineskip±.2\baselineskip}   
  {\itshape}  
  {0pt}       
  {\bfseries} 
  {.}         
  {5pt plus 1pt minus 1pt} 
  {(\thmnumber{#2})}          
\theoremstyle{assert}
\newtheorem{as}[cas]{}

\def\aqedsymbol{\ifmmode$\lrcorner$\else{\unskip\nobreak\hfil
\penalty50\hskip1em\null\nobreak\hfil$\lrcorner$
\parfillskip=0pt\finalhyphendemerits=0\endgraf}\fi} 
\newcommand{\aqed}{\renewcommand{\qed}{\aqedsymbol}}

\DeclareMathOperator{\mP}{\mathbb{P}}
\DeclareMathOperator{\mE}{\mathbb{E}}
\DeclareMathOperator{\brn}{br}
\DeclareMathOperator{\bin}{Bin}

\newcommand{\tih}{\tilde{h}}

\newcommand{\floor}[1]{\left\lfloor{#1}\right\rfloor}

\date{}
\newcounter{num}

\begin{document}
\title{Bipartite independence number in graphs with bounded maximum degree}
\date{\vspace{-5ex}}
\author{
    Maria Axenovich
    \thanks{
    	Karlsruhe Institute of Technology, Karlsruhe, Germany;
        email:
        \mbox{\texttt{maria.aksenovich@kit.edu}}.
	}     
    \and
    Jean-S{\'e}bastien Sereni
    \thanks{
        Centre National de la Recherche Scientifique, ICube (CSTB), Strasbourg, France;
        email: \mbox{\texttt{sereni@kam.mff.cuni.cz}}.
    }
    \and
     Richard Snyder
     \thanks{
     	Karlsruhe Institute of Technology, Karlsruhe, Germany;
	email: \mbox{\texttt{richard.snyder@kit.edu}}.
	}
	\and
	Lea Weber
	\thanks{
	Karlsruhe Institute of Technology, Karlsruhe, Germany;
	email: \mbox{\texttt{lea.weber@kit.edu}}.
	}
}

\maketitle
\singlespace

\begin{abstract}
    \setlength{\parskip}{\medskipamount}
    \setlength{\parindent}{0pt}
    \noindent
    We consider a natural, yet seemingly not much studied,
    extremal problem in bipartite graphs. A \emph{bi-hole} of size~$t$ in a bipartite graph $G$ is a copy
    of~$K_{t, t}$ in the bipartite complement of $G$. Let~$f(n, \Delta)$ be the largest~$k$ for which
    every $n \times n$ bipartite graph with maximum degree $\Delta$ in one of the parts has a bi-hole of
    size~$k$. Determining $f(n, \Delta)$ is thus the bipartite analogue of finding the largest
    independent set in graphs with a given number of vertices and bounded maximum degree. Our main result
    determines the asymptotic behavior of $f(n, \Delta)$. More precisely, we show that for large but
    fixed $\Delta$ and $n$ sufficiently large, $f(n, \Delta) = \Theta(\frac{\log \Delta}{\Delta} n)$.  We
    further address more specific regimes of $\Delta$, especially when $\Delta$ is a small fixed
    constant. In particular, we determine~$f(n, 2)$ exactly and obtain bounds for $f(n, 3)$, though
    determining the precise value of~$f(n, 3)$ is still open.
 \end{abstract}

\section{Introduction}\label{sec:intro}

The problem of finding~$g(n, \Delta)$, the smallest possible  size of a largest independent set in an
$n$-vertex  graph with given maximum degree~$\Delta$ is not very difficult. Indeed, one can consider the
graph that is the disjoint union of~$\lfloor n/(\Delta+1)\rfloor$ complete graphs on~$\Delta+1$ vertices
each and a complete graph on the remaining vertices. This shows that $g(n,\Delta) \leq
\lceil n/(\Delta + 1)\rceil$. On the other hand, every $n$-vertex graph of maximum degree~$\Delta$
contains an independent set of size~$\lceil n/(\Delta+1)\rceil$, obtained for example by the greedy
algorithm. Consequently, $g(n,\Delta)=\lceil n/(\Delta+1)\rceil$.  The situation is more interesting for
regular graphs, see Rosenfeld~\cite{Ros} for a more detailed analysis. The analogous problem in the
bipartite setting is more complex: determining the smallest possible `bipartite independence number' of a
bipartite graph with maximum degree $\Delta$ is still unresolved, even for $\Delta = 3$. To make this
precise, we shall start with a few definitions.

A subgraph of the complete bipartite graph~$K_{n,n}$ with~$n$ vertices in each part is called an $n\times
n$ bipartite graph.  A \emph{bi-hole} of size~$k$ in a bipartite graph~$G = (A \cup B, E)$  with a given
bipartition~$A, B$, is a pair~$(X, Y)$ with $X \subseteq A$, $Y \subseteq B$ such that $|X| = |Y| = k$,
and such that there are no edges of~$G$ with one endpoint in $X$ and the other endpoint in~$Y$.  
Thus, the size of the largest
bi-hole can be viewed as a bipartite version of the usual independence number.  This work is devoted to
studying the behavior of this function.

For a graph $G$, we denote the degree of a vertex $x$ by $\deg_G(x)$ or $\deg(x)$, the number of edges by $e(G)$,  the number of vertices by $|G|$, and the maximum degree by $\Delta(G)$. We write~$\log$ for the natural logarithm.

\begin{definition}
Let
$f(n, \Delta)$ be the largest integer $k$ such that any $n \times n$ bipartite graph $G = (A \cup B, E)$
with $\deg(a) \le \Delta$ for all~$a \in A$ contains a bi-hole of size $k$.  Let $f^*(n, \Delta)$ be  the largest~$k$ such that any
$n \times n$ bipartite graph~$G$ with~$\Delta(G) \le \Delta$ contains a bi-hole of size~$k$. 
\end{definition}

While $f(n, \Delta)$ is defined by restricting the maximum degree in one part of the graph,  $f^*(n, \Delta)$ is its  `symmetric' version.
Observe that  $f(n, \Delta) \le f^*(n, \Delta)$ for any natural numbers~$n$ and~$\Delta$ for which these functions are defined. 

\begin{theorem}\label{thm:lower}
There exists an integer~$\Delta_0$  and a positive constant $c$ such that the following holds.
    For any~$\Delta \ge \Delta_0$ there is~$N_0=N_0(\Delta)\ge5\Delta\log\Delta$ such that for
    any~$n>N_0$, \[
	f(n, \Delta) \ge \frac12\cdot\frac{\log \Delta}{\Delta} n.  \]
	In addition, $f^*(n, \Delta) \geq c\frac{\log \Delta}{\Delta} n$.
\end{theorem}

\begin{theorem}\label{thm:upper}
    Let~$\Delta$ be an integer, $\Delta \geq 27$. If~$n\ge \frac{\Delta}{\log\Delta}$,
    then \[
        f(n,\Delta) \le 8\cdot\frac{\log\Delta}{\Delta}n.  \]
\end{theorem}

\noindent
Note that the expression in the upper bound provided by \Cref{thm:upper} is trivial if~$\Delta<27$ since
then $8\cdot\frac{\log\Delta}{\Delta}\geq 1$. 

Theorems~\ref{thm:upper} and~\ref{thm:lower} thus determine~$f(n, \Delta)$ asymptotically for
sufficiently large, but fixed~$\Delta$ and growing~$n$. We state this concisely in the following
corollary.
\begin{corollary}\label{cor:main}
There exists an integer~$\Delta_0$ such that the following holds.  For any~$\Delta \ge \Delta_0$ there
    is~$N_0=N_0(\Delta)$ such that for any~$n>N_0$, \[
    \frac12\cdot\frac{\log \Delta}{\Delta} n \le f(n, \Delta) \le 8\cdot\frac{\log \Delta}{\Delta}n.  \]
\end{corollary}

\begin{corollary} There is a $\Delta_0$ such that if $\Delta>\Delta_0$, then 
$f^*(n, \Delta) = \Theta(\frac{\log \Delta}{\Delta}n)$. 
\end{corollary}

\noindent
Note that \Cref{thm:lower} does not cover the entire range of \Cref{thm:upper}: one wonders
if~$f(n,\Delta)$ is also~$\Theta(\frac{\log\Delta}{\Delta}n)$ when, for example, $n$ is in the interval $(\Delta/\log \Delta,\,\, \Delta \log\Delta)$. We leave this as an open problem, see \Cref{conclusions}.\\

Given \Cref{cor:main}, it is natural to consider the behavior of~$f(n, \Delta)$ when~$\Delta$ is small.
In this case, we have only the following modest results. The bounds are obtained as corollaries of
general bounds  that  become less and less precise as~$\Delta$ grows; see Table~\ref{table:small-delta} and \Cref{sec:small-delta} for more
explicit values.

\begin{table}\label{table:small-delta}
\centering
        \begin{tabular}{@{\extracolsep{4pt}}cccc}
            $\Delta$ & Lower bound & Upper bound \\
\hline
            $3$ & $0.34116$ & $0.4591$\\
            $4$ & $0.24716$ & $0.4212$\\
            $5$ & $0.18657$ & $0.3887$\\
            $6$ & $0.14516$ & $0.3621$\\
            $7$ & $0.11562$ & $0.3395$\\
            $8$ & $0.09384$ & $0.3201$\\
            $9$ & $0.07735$ & $0.3031$\\
            $10$ & $0.06459$ & $0.2882$\\
        \end{tabular}
    \caption{Explicit asymptotic lower and upper bounds on~$f(n,\Delta)$, divided by~$n$, obtained for small values of~$\Delta$, for~$n$
    large enough.
}
\end{table}

\begin{theorem}\label{thm:max-deg-small}
For any $\Delta \geq 2$ and any $n\in \mathbb{N}$ we have $f(n, \Delta)\geq \lfloor \frac{n
    -2}{\Delta}\rfloor $.  Moreover, for any  $n\in \mathbb{N}$ we have $f(n, 2) = \lceil n/2\rceil - 1$, and
    there exists~$n_0$ such that if~$n>n_0$, then $0.3411 n < f(n, 3)\le f^*(n,3) < 0.4591n$.
\end{theorem}

We also consider the other end of the regime for~$\Delta$, when $\Delta$ is close to $n$. In particular,
when $\Delta$ is linear in $n$, say $\Delta = n - cn$, it follows from \Cref{thm:upper} that $f(n, n -
cn) = O(\log n)$. Furthermore, it is not too difficult to show that $f(n, n - cn) = \Omega(\log n)$ (see
\Cref{prop:linear-delta} for details). When $\Delta$ is much larger (i.e., $\Delta = n - o(n)$), bounding
$f(n, \Delta)$ bears a strong connection to the Zarankiewicz problem, and we are able to obtain the
following result.  We formulate it in terms of a bound on the degrees guaranteeing  a bi-hole  of a
constant size $t$.  Let \[\Delta(t) \coloneqq \max\{ q\,:\,  f(n, q) =t\}.\]

\begin{theorem}\label{thm:large-random}
Let $t \ge 4$ be an integer. There is a positive
    constant~$C$ and an integer~$N_0$ such that if $n > N_0$, then 
 $ n - Cn^{1 - 1/t}   \leq \Delta(t)  \leq      n - Cn^{1 - \frac{2}{t+1}}$. 
 In addition there is an integer~$N_0$, such that if $n>N_0$, then
 $\Delta(2) = n - n^{1/2}(1+o(1)) $ and $\Delta(3) = n -  n^{2/3}(1+o(1))$. 
 \end{theorem}
 
The rest of the paper is structured as follows. We describe connections between the function~$f(n,
\Delta)$,  classical bipartite Ramsey numbers, and the Erd\H{o}s-Hajnal conjecture in
Section~\ref{sec:related}.  We prove Theorems~\ref{thm:lower} and~\ref{thm:upper}  in \Cref{proof1} and
prove Theorem~\ref{thm:max-deg-small} and establish the values for Table~\ref{table:small-delta} in
\Cref{proof2}. We prove Theorem~\ref{thm:large-random} in \Cref{sec:large-delta}.  \Cref{conclusions}
provides concluding remarks and open questions.

\section{Related problems}\label{sec:related}

The function~$f(n, \Delta)$ is closely related to the bipartite version of the
Erd\H{o}s-Hajnal conjecture, bipartite Ramsey numbers, and the Zarankiewicz function.

\bigskip

A conjecture of
Erd\H{o}s and Hajnal~\cite{EH} asserts that for any graph~$H$ there is a
constant $\epsilon>0$ such that any $n$-vertex graph that does not contain~$H$
as an induced subgraph has either a clique or a coclique on at
least~$n^{\epsilon}$ vertices.  While this conjecture remains open, (see, for
example, the survey by Chudnovsky~\cite{Chu}) the bipartite version of the
problem has been considered.  For a bipartite graph~$G$, let~$\tilde{h}(G)$ be
the size of a largest \emph{homogeneous set} in~$G$ respecting sides, i.e.\ the largest
integer~$t$ such that~$G$ either contains a bi-hole of size~$t$ or a complete
bipartite subgraph with~$t$ vertices in each part. Given a bipartite graph~$H$, let~$\tih(n, H)$ be the smallest value
of~$\tilde{h}(G)$ over all $n\times n$ bipartite graphs~$G$ that do not
contain~$H$ as an induced subgraph respecting sides.  It is implicit from a
result of Erd\H{o}s, Hajnal, and Pach~\cite{EHP} that if~$H$ is a bipartite graph
with the smallest part of size~$k$, then $\tih(n, H)=\Omega(n^{1/k})$.  A standard
probabilistic argument shows that if~$H$ or its bipartite complement contains a
cycle, then $\tih(n, H) = O(n^{1-\epsilon})$ for a positive~$\epsilon$.
Axenovich, Tompkins, and Weber~\cite{ATW} proved that~$\tih(n,H)$ is linear
in~$n$ for all but four strongly acyclic bipartite graphs~$H$ (a bipartite graph~$H$ being
\emph{strongly acyclic} if neither~$H$ nor its bipartite complement contains a
cycle).  Kor\'andi, Pach, and Tomon~\cite{KPT} independently obtained similar
results in the context of matrices.  Note that~$f(n, \Delta)$ for~$\Delta$ sublinear in~$n$ corresponds to~$\tih(n, K_{1, \Delta +1})$.
Indeed, a bipartite graph not having a star with~$\Delta +1 $ leaves (respecting sizes) is a graph with
vertices in one part having degrees at most~$\Delta$.  In such a graph there are clearly no complete
bipartite graphs with~$\Delta+1$ vertices in each part, so the largest homogeneous set is a bi-hole,
as~$\Delta$ is sublinear in~$n$. Its size is thus determined by~$f(n, \Delta)$.

\bigskip

Furthermore, the parameter~$f(n, \Delta)$ bears a connection to bipartite Ramsey numbers.  If~$H_1$
and~$H_2$ are bipartite graphs, then the \emph{bipartite Ramsey number} $\brn(H_1, H_2)$ is the smallest
integer~$n$ such that any red-blue coloring of the edges of~$K_{n, n}$ produces a red copy of~$H_1$ or a
blue copy of~$H_2$ respecting sides. Thus, if~$f(n, \Delta)=k$, then $\brn(K_{1, \Delta+1}, K_{k,k})=n$.
For results on bipartite Ramsey numbers, see Caro and Rousseau~\cite{caro-rousseaux},
Thomason~\cite{thomason}, Hattingh and Henning~\cite{hattingh-henning},  Irving~\cite{irving}, and
Beineke and Schwenk~\cite{beineke-schwenk}.

\bigskip

Finally, considering the bipartite complement, determining $f(n, \Delta)$ is related to the Zarankiewicz
problem in bipartite graphs. We let $z(n, t)$ denote the maximum number of edges in a
subgraph of $K_{n, n}$ with no copy of $K_{t, t}$. Finding a large bi-hole in a bipartite graph is the
same as finding a large copy of $K_{t, t}$ in the bipartite complement, where the bipartite complement
has large minimum degree on one side (this is spelled out more carefully in \Cref{sec:large-delta}).
There is some literature on the Zarankiewicz problem for~$t$ large (see, for example,  Balbuena et
al.~\cites{BGMV, BGMV-1}, \v{C}ul\'ik~\cite{Cul}, F\"uredi and Simonovits~\cite{FS}, Griggs and
Ouyang~\cite{GO}, and Griggs, Simonovits, and Thomas~\cite{GST}). However, most of these results address
the case when~$t$ is close to $n/2$, or when the results do not lead to improvements on our bounds.
 

\section{Bounds on \texorpdfstring{$f(n, \Delta)$}{f(n,Delta)}}\label{sec:bounds}
In this section we establish upper and lower bounds on $f(n, \Delta)$ for 
various ranges of $\Delta$.
First, we establish the exact value of $f(n, 2)$ that is used in other results. 
Then, we treat the case when $\Delta$ is fixed but large. We then move on to consider when~$\Delta$ is
a small fixed constant, and finally, when $\Delta$ is large, close to $n$.

\begin{lemma}\label{lem:max-deg-2}
For every positive integer~$n$, we have $f(n, 2) = \lceil n/2\rceil - 1$.
\end{lemma}
\begin{proof}
To see that $f(n, 2)<n/2$, simply consider an even cycle~$C_{2n}$ on~$2n$ vertices. It remains to
    establish the lower bound.  Let $H= (A\cup B, E)$ be a bipartite graph with $n$ vertices in each part
    and the degree of each vertex in~$A$ is at most~$2$. Note that we may assume without loss of
    generality that the degree of each vertex in~$A$ is exactly~$2$. Consider the auxiliary graph~$G$
    with vertex set~$B$, in which two vertices are adjacent if and only if
    they have a common neighbor in~$H$. Consequently, there is a natural bijection between the edges of~$G$ and~$A$,
    and thus~$G$ has $n$ vertices and~$n$ edges.
    We assert that~$G$ contains a set~$E'$ of edges and a set~$V'$ of vertices each of
    size~$\lceil n/2\rceil -1$, and such that no edge in~$E'$ has a vertex in~$V'$. Note, then, that this pair of sets corresponds to
    parts of a bi-hole in~$H$ of size $\lceil n/2 \rceil - 1$, thus proving that~$f(n,2)\ge \lceil n/2\rceil -1$. The rest of the proof is devoted to proving the above assertion.

    To this end, we consider the components of~$G$: a component~$C$ is \emph{dense} if
    $e(G[C]) \ge |C|$.  Let~$S_1, \dotsc, S_k$ be the components of $G$, enumerated such that $S_1, \dotsc, S_m$ are dense
    and the others are not.  Note that we must have at least one dense component, so~$m\ge1$, and it could be that all
    components are dense.  Let~$x$ be the number of components of $G$ that are not dense, that is, $x\coloneqq
    k-m\in\{0,\dotsc,k-1\}$.  Let~$v$ and~$e$ be the number of vertices and edges, respectively, in the union of all dense
    components of~$G$.  Then the total number of edges in non-dense components of~$G$ is~$n-e$ and the total number of vertices
    in these components is~$n-v$. In addition, the number of vertices in non-dense components is at least  the number of
    edges plus the number of components. Thus, $n-v\geq n-e+x$, so~$x\leq e-v$.

Let~$G'$ be a subgraph of~$G$ with precisely $\lceil n/2\rceil -1$ edges and consisting of $S_1, \dotsc,
    S_q$ and a connected subgraph of~$S_{q+1}$, for some $q\in\{0,\dotsc,k-1\}$.  In particular, if~$S_1$
    has at least $\lceil n/2\rceil -1$ edges, then $G'$ is a connected subgraph of $S_1$.  It suffices to
    show that~$G'$ has at most~$\lfloor n/2\rfloor+1$ vertices, since we can then choose a set~$V'$
    of~$\lceil n/2\rceil -1$ vertices in~$V(G)\setminus V(G')$, which along with~$E'\coloneqq E(G')$ will
    form the sought pair~$(V',E')$.  To this end, first notice that if~$G'$ has at most one non-dense
    component, then the number of vertices of~$G'$ is at most~$|E'|+1$, which is at most~$\lceil
    n/2\rceil\le\lfloor n/2\rfloor +1$, as desired.  Suppose now that~$G'$ has more than one non-dense
    component. It follows that~$G'$ contains all dense components of~$G$.  Let~$x'$ be the number of
    non-dense components of~$G'$. Then $x'\leq x$.  The number of edges in dense components of~$G'$ is
    $e$, thus the number of edges in non-dense components of $G'$ is $\lceil n/2\rceil -1-e$.  This
    implies that the number of vertices in non-dense components of~$G'$ is at most $(\lceil n/2\rceil
    -1-e)+x' \leq (\lceil n/2\rceil -1-e)+x$.  Adding the number~$v$ of vertices in dense components
    of~$G$ and the number of vertices in non-dense components of~$G'$, we see that the total number of
    vertices in~$G'$ is at most $ v + ( (\lceil n/2\rceil -1-e)+x) \leq \lceil n/2\rceil - 1\leq \lfloor
    n/2\rfloor +1$.   This concludes the proof.
\end{proof}

\subsection{Proof of Theorems \texorpdfstring{\ref{thm:lower} and~\ref{thm:upper}}{1.1 and 1.2}}\label{proof1}

The upper bound given in \Cref{thm:upper} comes from suitably modifying the random bipartite graph
$G\left(n, n, \frac{\Delta}{n}\right)$. The idea of the proof of the lower bound given in
\Cref{thm:lower} is as follows. Let~$G = (A \cup B, E)$ be an $n \times n$ bipartite graph with $\deg(x) \le
\Delta$ for every $x \in A$.  We choose an appropriate parameter $s$ and choose a subset $S$ of~$B$
uniformly at random from the set of all $s$-element subsets of $B$ and consider the set $T$ of vertices
in $A$ that have at least $\Delta - 2$ neighbors in $S$.  \Cref{lem:max-deg-2} can then be applied to the
bipartite graph induced on parts $(T,B\setminus S)$, as in this bipartite graph every vertex in~$T$ has
degree at most~$2$. Intuitively, the set $T$ should be ``large enough'' to guarantee a large bi-hole
in~$G$.  Floors and ceilings, when not relevant, are ignored in what follows. We start by establishing
the lower bound, that is \Cref{thm:lower}.

\begin{proof}[Proof of \Cref{thm:lower}]
Consider an arbitrary bipartite graph with parts~$A$ and~$B$ each of size~$n$ so that the degrees of
    vertices in~$A$ are at most~$\Delta$.  Choose a subset~$S$ of~$B$ of size~$(1-2x)n-2$ randomly and
    uniformly among all such subsets, where $x \coloneqq \frac{1}{2} \frac{\log \Delta}\Delta$.  We
    assume that $n \ge N_0\ge 5\Delta\log\Delta$ and~$\Delta\ge\Delta_0$ chosen large enough to satisfy the
    last inequality in the proof.  Let~$X$ be the random variable counting the number of vertices in $A$
    with at least $\Delta - 2$ neighbors in $S$.  Then \[\mathbb{E}X \ge n\cdot h(x, n, \Delta),\] where 
    \[h(x, n, \Delta) = \binom{\Delta}{\Delta-2} \binom{n-\Delta}{(1-2x)n -  \Delta}\binom{n}{(1-2x)n-2}^{-1}.\]

Observe that if $\mathbb{E}X \ge 2xn+2$, then there is a set $A'$ of at least $2xn+2$ vertices in $A$,
    each sending at most $2$ edges to $B \setminus S$. Since $|B\setminus S|=2xn+2$, \Cref{lem:max-deg-2}
    implies that there is a bi-hole between $A'$ and $B\setminus S$ of size at least $xn$. So, it is
    sufficient to prove that $h(x, n, \Delta) \geq 2x+2/n$.  Let us now verify this inequality.
    
    Recall that $x = \frac{1}{2} \frac{\log \Delta}\Delta$.   Let 
    $\alpha=1 - 2x$, so $\alpha =  1 - \frac{\log \Delta}\Delta = \frac{\Delta- \log \Delta}{\Delta} \in (0,1)$.  Note that~$\alpha n \ge\Delta$ since~$n\ge5\Delta\log\Delta$. Let $\beta =\frac{1}{\alpha} - 1$.  Then $\beta = \frac{\log \Delta}{\Delta - \log \Delta}$. We have
\begin{align}
	h(x, n, \Delta) &= \frac{\binom{\Delta}{\Delta-2} \binom{n-\Delta}{(1-2x)n - \Delta } }{ \binom{n}{(1-2x)n-2}} \notag \\
          &= \binom{\Delta}{2}\prod_{j = 2}^{\Delta - 1} \left(\frac{\alpha n - j}{n - j}\right) \cdot \left[\frac{(2xn+2)(2xn + 1)}{n(n - 1)}\right] \notag\\
	&> \binom{\Delta}{2}\prod_{j = 2}^{\Delta - 1} \left(\frac{\alpha n - j}{n - j}\right) \cdot (2x)^2  \notag \\
	&=   \binom{\Delta}{2}(2x)^2\alpha^{\Delta - 2}\prod_{j = 2}^{\Delta - 1} \left( 1 - \frac{\beta j}{n - j}\right) \notag \\
	&\geq  \binom{\Delta}{2}(2x)^2\alpha^{\Delta - 2}\left(1 - \frac{\beta \Delta}{n - \Delta}\right)^{\Delta - 2} \label{increase}  \\
	&\ge \binom{\Delta}{2}(2x)^2\alpha^\Delta\left(1 - \frac{\beta\Delta}{n-\Delta}\right)^\Delta \label{increase'},
\end{align}
where~\eqref{increase} holds because the function $j \mapsto \frac{\beta j}{ n - j}$ is increasing, as
    $\beta > 0$. Now, expressing~$\beta$ in terms of~$\Delta$, we note that
\[
    \frac{\beta \Delta}{n  - \Delta} \leq \frac{ \Delta \log
    \Delta/(\Delta - \log \Delta) }{5\Delta\log\Delta - \Delta}  = \frac{\log \Delta}{(5\log\Delta - 1)(\Delta- \log \Delta)} \leq 1,
\]
and therefore Bernoulli's inequality can be applied to~\eqref{increase'}. It follows that
\begin{align}
h(x, n, \Delta)	&> \binom{\Delta}{2}(2x)^2 (1 - 2x)^{\Delta}\left( 1 - \frac{\Delta^2\log\Delta}{(\Delta - \log\Delta)(n - \Delta)}\right)  \notag\\
	&\geq \binom{\Delta}{2}(2x)^2 (1 - 2x)^{\Delta} \left( 1 - \frac{4\Delta\log\Delta}{n}\right) \label{e7}\\
	&\geq \binom{\Delta}{2}(2x)^2 (1 - 2x)^{\Delta} \frac{1}{5} \label{e8},
\end{align}	
where~\eqref{e7} follows since $\frac{1}{n - \Delta} < \frac{2}{n}$ and $\log \Delta < \Delta/2$,
    and~\eqref{e8} holds since $n > 5\Delta\log\Delta$. Now, note that $(1 - 2x)^\Delta = \left(1 -
    \frac{\log\Delta}{\Delta}\right)^\Delta \ge \frac{1}{2}e^{-\frac{\log\Delta}{\Delta}\cdot \Delta} =
    \frac{1}{2\Delta}$. Thus, from~\eqref{e8} we obtain
\[
    h(x, n, \Delta) > \binom{\Delta}{2}(2x)^2\frac{1}{10\Delta} = (2x) \frac{(\Delta - 1)\log \Delta}{20\Delta}.
\]
Finally, to bound the right-hand side of the above inequality from below, observe that
\[
(2x)\frac{  (\Delta-1) \log \Delta }{ 20 \Delta} \ge (2x)\left(1 + \frac{1}{40}\log \Delta\right) = 2x + \frac{\log^2\Delta}{40\Delta} \ge 2x + \frac{2}{5\Delta\log\Delta} \ge 2x + \frac{2}{n},
\]
where these inequalities hold for sufficiently large $\Delta$. Accordingly, $h(x, n, \Delta) > 2x + \frac{2}{n}$, which concludes the proof of \Cref{thm:lower}.
\end{proof}

To prove \Cref{thm:upper} we shall need to use Chernoff's bound.  Specifically, we use the following
version (see~\cite{random-graphs}*{Corollary~21.7, p.401}, for example).

\begin{lemma}[Chernoff's bounds]\label{lem:chernoff}
  Let $X$ be a random variable with distribution $\bin(N, p)$ and $\varepsilon \in (0, 1)$. Then
    \begin{align}
        \mP[X \ge (1 + \varepsilon)\mE X] & \le
        \exp\left(-\frac{\varepsilon^2}{3}\mathbb{E}X\right)\quad\text{and}\label{cher-up}\\
        \mP[X \le (1 - \varepsilon)\mE X] & \le 
        \exp\left(-\frac{\varepsilon^2}{2}\mE X\right).\label{cher-low}
    \end{align}
\end{lemma}

\begin{proof}[Proof of \Cref{thm:upper}]
Let $\Delta\ge27$.  Suppose  that~$n\ge\tfrac{3\Delta}{20\log(\Delta/2)}$.  Set $N\coloneqq2n$ and $\Delta' \coloneqq  \Delta/2$,
    so in particular~$\Delta'\ge13.5$.  We consider first $H \coloneqq G\left(N, N,
    \frac{\Delta'}{N}\right)$, that is, $H$ is a random bipartite graph with parts $A$ and~$B$ each of
    size $N$, where each edge $ab$ with $a \in A$ and $b \in B$ is chosen independently and uniformly at
    random with probability $\Delta'/N$. We first establish that the random
    graph~$H$ contains no ``large'' bi-holes with fairly large probability.  In the following, for subsets~$X \subseteq A$ and~$Y \subseteq
    B$, let~$e(X, Y)$ denote the number of edges with one endpoint in~$X$ and the other in~$Y$.

\begin{as}\label{claim:no-large-holes}
    With probability at least~$0.75$, any two subsets $X \subset A$ and $Y \subset B$ with $|X| = |Y| =
\frac{2N\log \Delta'}{\Delta'}$ satisfy $e(X, Y) > 0$.
\end{as}
\begin{proof}
Set $m \coloneqq \frac{2N\log \Delta'}{\Delta'}$ and note that
    $m$ is therefore at least~$\frac65$.
    Suppose that $X \subset A$ and $Y \subset B$ both
    have size~$m$. Then
     \[ \mP(e(X, Y)=0)=  \left(1 - \frac{\Delta'}{N}\right)^{m^2}.\]  
     Let $p$ be the probability that there is a pair~$(X, Y)$, with $X \subset A$ and $Y \subset B$, $|X|=|Y|=m$, such that $e(X,Y)=0$.
    Forming a union bound over all possible pairs of sets of size $m$, we have \[
	p\leq \binom{N}{m}^2\left(1 - \frac{\Delta'}{N}\right)^{m^2}\leq \left(\frac{Ne}{m}\right)^{2m} e^{-\frac{\Delta'm^2}{N}} = \left(\frac{\Delta' e}{2\log
    \Delta'} e^{-\log \Delta'}\right)^{2m} = {\left(\frac{e}{2\log \Delta'}\right)}^{2m}\leq 0.25.  \]
Here, we used  the standard estimates $\binom{t}{k} \le \left(\frac{t\cdot e}{k}\right)^k$,  $1 - x \le e^{-x}$, 
the fact  that $(e/2\log \Delta') < 0.53$ because $\Delta' \ge 13.5$, as well as inequality $2m\ge\frac{12}{5}$.
 This establishes~\eqref{claim:no-large-holes}. 
\aqed
\end{proof}

Next we show that, with probability sufficiently large for our purposes, at least half of the vertices of
    $H$ have degree at most $\Delta'+ \sqrt{3\Delta'}$. 

\begin{as}\label{claim:no-large-degrees}
With probability greater than~$0.25$, the number of vertices~$v\in A$ with more than~$\Delta' +  \sqrt{3\Delta'}$ neighbors
in~$B$ is at most~$N/2$.
\end{as}

\begin{proof}
We use standard concentration inequalities to show that the degree of every vertex in $A$ is
approximately~$\Delta'$.  For each vertex $v \in A$, let $X_v$ be the degree of $v$ in $H$. Noting that
    $\mathbb{E}X_v = \Delta'$, we apply \eqref{cher-up} from \Cref{lem:chernoff} with 
    $\varepsilon \coloneqq \sqrt{3/\Delta'} < 1$ to obtain \[
    \mathbb{P}[X_v \ge (1 + \varepsilon)\Delta'] \le \exp\left(-\frac{{(\sqrt{3/\Delta'})}^2}{3}\Delta'\right)
    = e^{-1}.  \]
Letting $X$ be the random variable counting those vertices $v \in A$ with $X_v \ge (1 +
\varepsilon)\Delta'$, we observe that $\mathbb{E}X \le e^{-1}N$, and therefore, by Markov's inequality, we
        deduce that $\mathbb{P}[X \ge N/2] \leq \frac2e < 0.75$, thereby establishing~\eqref{claim:no-large-degrees}.  \aqed
\end{proof}

It follows from~\eqref{claim:no-large-holes} and~\eqref{claim:no-large-degrees} that with positive
    probability, $H$ has no large bi-holes, and at least half of the vertices in $A$ have degree at most
    $\Delta' + \sqrt{3\Delta'}\leq 2\Delta' \leq \Delta$.  We now fix
    such a graph~$H$.  We can thus choose a subset~$A'$ of~$A$ of size~$N/2=n$ such that every vertex
    in~$A'$ has degree at most~$\Delta$ in~$H$. Now, arbitrarily choosing a subset~$B'$ of~$B$ of
    size~$n$, we know that the subgraph~$H'$ of~$H$ induced by~$A'\cup B'$ is an $n\times n$ bipartite
    graph with maximum degree~$\Delta$ on one side and without bi-hole of size larger than \[
	2(2n)\left(\frac{\log(\Delta/2)}{(\Delta/2)}\right) 
        < 8n\left(\frac{\log \Delta}{\Delta}\right).  \]

 In order to obtain a lower bound on $f^*(n, \Delta)$, all that is required is to make the
example obtained above have bounded maximum degree. Thus, it suffices to apply Chernoff's inequality to
all vertices (instead of just the vertices in $A$). We may have to remove more vertices after doing this,
but the loss will only be reflected in the constant.
This completes the proof of \Cref{thm:upper}.
\end{proof}


\subsection{Bounding \texorpdfstring{$f(n, \Delta)$}{f(n,Delta)} for small \texorpdfstring{$\Delta$}{Delta}}\label{sec:small-delta}\label{proof2}

We have already established a part of \Cref{thm:max-deg-small} via \Cref{lem:max-deg-2}, namely,
we showed that $f(n, 2) = \lceil n/2\rceil - 1$. Our aim in this section is to investigate the behavior
of $f(n, 3)$ more closely, and to complete the proof of \Cref{thm:max-deg-small}. First, let us note the
following lower bound on $f(n, \Delta)$, valid for all integers~$n$ and~$\Delta$ greater than~$1$.

\begin{proposition}\label{prop:general-delta}
    If~$n$ and~$\Delta$ are two integers greater than~$1$, then
    $f(n,\Delta) \ge \floor{\frac{n - 2}{\Delta}}$.
\end{proposition}

\begin{proof}
We shall prove this by induction on~$\Delta$ with the base case $\Delta=2$ following from
    \Cref{lem:max-deg-2}. Let~$H= (A\cup B, E)$
    be a bipartite graph with~$n$ vertices in each part and such that the degree of each vertex in
    part~$A$ is equal to $\Delta$, $\Delta\geq 3$.

Consider a set~$X$ of $\lfloor(n - 2)/ \Delta\rfloor$ vertices in $B$. If $|N(X)|\leq n-
    \lfloor(n-2)/\Delta\rfloor$, then~$X$ and~$A\setminus N(X)$ form a bi-hole with at least
    $\lfloor(n-2)/\Delta\rfloor$ vertices in each part.  Otherwise, $|N(X)| > n - \lfloor(n -
    2)/\Delta\rfloor$.  Let~$G' \coloneqq G[N(X) \cup (B\setminus X)]$. Then each of the parts of~$G'$ has
    size at least~$n- \lfloor{(n - 2)/\Delta}\rfloor \ge n - (n-2)/\Delta$ and the maximum degree of
    vertices of~$N(X)$ in~$G' $ is at most~$\Delta -1$. Thus, by induction $G'$ has a bi-hole of size at
    least $\left\lfloor\frac{1}{\Delta-1} (n - (n-2)/\Delta - 2)\right\rfloor = \left\lfloor\frac{n -
    2}{\Delta}\right\rfloor$.
\end{proof}

It follows from the above proposition that $f(n, 3) \ge \floor{(n - 2)/3}$. However, this lower
bound can be improved slightly by choosing a random subset of $B$ and considering the neighborhood of
this set in $A$, similarly as in the proof of the lower bound in \Cref{thm:lower}.

\begin{lemma}\label{lem:recursion}
If~$n$ and~$\Delta$ are two integers greater than~$1$, then
$f(n, \Delta) \geq f(\lfloor\xi n\rfloor, \Delta-1)$, where 
$\xi=\xi(\Delta)$ is a solution to the inequality $1- \xi^\Delta \ge \xi$.
\end{lemma}

\begin{proof}
For simplicity we omit floors in the following. Let~$G$ be a bipartite graph with parts~$A$ and~$B$ each of size~$n$ such that the vertices in~$A$ have
    degrees at most~$\Delta$.  We shall show that there is a set~$S\subset B$, such that $|S|= (1-\xi)n$
    and such that $|N(S)| \geq \xi n$.  To do this, we shall choose~$S$ randomly and uniformly out of all
    subsets of~$B$ of size~$(1-\xi)n$ and show that the expected number~$X$ of vertices from~$A$ with at
    least one neighbor in~$S$ is at least $\xi n$.  Indeed, if~$p$ is the  probability  for a fixed
    vertex in~$A$ not to have a neighbor in $S$, then \[ 
	p = \frac{\binom{n- \Delta}{(1-\xi)n}}{\binom{n}{(1-\xi)n}}.  \]
Using the identity $  \binom{n - l}{k}/ \binom{n}{k} =  \binom{n - k}{l}/\binom{n}{l}$, we see that $p =
    \binom{\xi n}{\Delta}/\binom{n}{\Delta}$. Now, using the inequality $\binom{\delta n}{r} \le \delta^r
    \binom{n}{r}$, which is valid for every $\delta \in (0, 1)$, we find that \[
	p = \frac{\binom{\xi n}{\Delta}}{\binom{n}{\Delta}} \le \xi^\Delta.  \]
Thus, $\mathbb{E}X = n(1-p) \geq n (1- \xi^\Delta) \ge n\xi$ by our choice of $\xi$. 
Consequently, with positive probability $|N(S)|\geq \xi n$. We also have $|B\setminus S| = \xi
n$. Since each vertex  of~$N(S)$ sends at most~$\Delta-1$ edges to~$B\setminus S$, it follows that there
is a bi-hole between~$N(S)$ and $B\setminus S$ of size $f(\xi n,  \Delta-1)$. This completes the proof.
\end{proof}

We make explicit some lower bounds obtained using \Cref{lem:recursion} (and \Cref{lem:max-deg-2}).
\begin{corollary}\label{cor-small-lower}
There exists~$N_0$ such that if~$n\ge N_0$, then
$f(n,3) > 0.34116n$, $f(n,4)> 0.24716n$, $f(n,5) > 0.18657n$, and   $f(n,6) > 0.14516n$.
\end{corollary}

The next natural step regarding small values of~$\Delta$ is to evaluate how good the bounds written in
\Cref{cor-small-lower} are. The following upper bounds are obtained by analysing the pairing model (also known as the \emph{configuration model}) to build random regular graphs, tailored to the bipartite setting.

\begin{lemma}\label{lem-pairing}
    Let~$\Delta$ be an integer greater than~$2$, and assume that~$\beta\in(0,1)$ is such that \[
    \frac{(1-\beta)^{2\Delta(1-\beta)}}{\beta^{2\beta}(1-\beta)^{2(1-\beta)}(1-2\beta)^{\Delta(1-2\beta)}}
    <1.  \] 
    Then there exists~$N_0=N_0(\beta)$ such that for every~$n>N_0$ we have $f(n, \Delta) \le f^*(n, \Delta) < \beta n$. In particular, for $n$ sufficiently large there exists a
$\Delta$-regular $n\times n$ bipartite graph with no bi-hole of size at least~$\beta n$.
\end{lemma}

\begin{proof} 
We shall work with the configuration model of Bollob{\'a}s~\cite{Bol81} suitably altered to produce a bipartite graph.  Fix an integer~$n$ and consider two sets
    of~$\Delta n$ (labelled) vertices each:~$X=\{x_1^1,\dotsc,x_1^\Delta,\dotsc,x_n^1,\dotsc,x_n^\Delta\}$ and
$Y=\{y_1^1,\dotsc,y_1^\Delta,\dotsc,y_n^1,\dotsc,y_n^\Delta\}$.
Choose a perfect matching  $F$ between $X$ and $Y$ uniformly at random. We call~$F$ a \emph{pairing}.

Given a pairing~$F$, for each~$i\in\{1,\dotsc,n\}$ the vertices~$x_i^1,\dotsc,x_i^\Delta$ are identified into a new
    vertex~$x_i$, and similarly the vertices~$y_i^1,\dotsc,y_i^\Delta$ are identified into a new vertex~$y_i$.
    This yields a multi-graph~$G'$.  We prove that with positive probability $G'$ is a simple graph. To see why this holds,
    first notice that the total number of different pairings is~$(\Delta\cdot n)!$.
    Second, each fixed (labelled) $\Delta$-regular $n\times n$ bipartite graph arises from precisely
    $(\Delta!)^{2n}$ different pairings (because for each vertex $x_i$ we can freely permute the vertices $\{x_i^1, \ldots, x_i^\Delta\}$).
    Third, McKay, Wormald and Wysocka~\cite{MWW04} proved that the
    number of different labelled $\Delta$-regular $n\times n$ bipartite graphs is \[
        (1+o(1))\exp\left(-\frac{{(\Delta-1)}^2}2\cdot((\Delta-1)^2+1)\right)\frac{(\Delta\cdot
        n)!}{({\Delta!)}^{2n}},
    \]
    which is at least $c \frac{(\Delta\cdot n)!}{(\Delta!)^{2n}}$ for some~$c>0$. 
    Combining these three facts, we find that $\mathbb{P}(\text{$G'$ is a graph}) \ge c > 0$,
     as announced.
        
    Now, fix~$k=k(n)=\beta\cdot n$ for some~$~\beta\in(0,1)$. 
    For each~$i\in\{1,\dotsc,n\}$, set~$X_i \coloneqq \{x_i^1,\dotsc,x_i^\Delta\}$ and~$Y_i \coloneqq
    \{y_i^1,\dotsc,y_i^\Delta\}$.  Fixing a family of~$k$ sets~$\mathcal{X} = \{X_{i_1},\dotsc,X_{i_k}\}$ and also $\mathcal{Y} = \{Y_{i_1},\dotsc,Y_{i_k}\}$, let $W(\mathcal{X}, \mathcal{Y})$ be the event that the union~$U$ of the sets in $\mathcal{X} \cup \mathcal{Y}$ spans no edge in $F$ (and thus corresponds to a bi-hole in $G'$). Let us now find the probability of $W(\mathcal{X}, \mathcal{Y})$.
    There are~$2\Delta k$ edges incident to vertices in~$U$. The number of ways to choose these~$2\Delta k$ edges
    is as follows: if the edge is incident to a vertex in~$\bigcup_{j=1}^k X_{i_j}$, then the other
    end-vertex must belong to~$Y\setminus\bigcup_{j=1}^kY_{i_j}$, and hence there are \[
        (\Delta n-\Delta k)\dotsc(\Delta n-2\Delta k+1)\]
    different ways of choosing the edges incident to a vertex in~$\bigcup_{j=1}^k X_{i_j}$.
    The situation is analogous for edges incident to~$\bigcup_{j=1}^k Y_{i_j}$, yielding a total of $(\Delta n-\Delta k)^2\dotsc(\Delta n-2\Delta k+1)^2$
    ways to choose the~$2\Delta k$ edges incident to a vertex in~$U$.  For each such choice there are $(\Delta n-2\Delta k)!$ ways to choose the remaining edges, for a total of $(\Delta n-2\Delta k)!\cdot(\Delta n-\Delta k)^2\dotsc(\Delta n-2\Delta k+1)^2$
    different pairings in which~$U$ spans no edge. It follows that 
    \[
        \mathbb{P}(W(\mathcal{X}, \mathcal{Y})) = \frac{(\Delta n-2\Delta k)!\cdot(\Delta n-\Delta k)^2\dotsc(\Delta n-2\Delta k+1)^2}{(\Delta n)!}=
    \frac{{((\Delta n-\Delta k)!)}^2}{(\Delta n)!(\Delta n-2\Delta k)!}.
\]

Let $W = \bigcup_{\mathcal{X}, \mathcal{Y}} W(\mathcal{X}, \mathcal{Y})$ be the event that~$F$ contains a
bi-hole of size~$k$. Taking a union bound over all $\binom{n}{k}^2$ choices of $(\mathcal{X},
    \mathcal{Y})$,  we find that
    \begin{equation}\label{eq-expected}
    \mathbb{P}(W) \le {\binom{n}{k}}^2\frac{{((\Delta n-\Delta k)!)}^2}{(\Delta n)!(\Delta n-2\Delta
        k)!}.
\end{equation}

Using Stirling's approximation, \[
    \sqrt{2\pi n}\left(\frac{n}{e}\right)^n\exp\left(-\frac{1}{12n+1}\right)\le n! \le
\sqrt{2\pi n}\left(\frac{n}{e}\right)^n\exp\left(-\frac{1}{12n}\right),\]
in~\eqref{eq-expected},
and ignoring the exponential factors (they can be bounded from above
by~$\exp(1/(cn))$ for some positive constant~$c$, and hence be made arbitrarily close
    to~$1$), we thus obtain \[
        {\binom{n}{k}}^2 
        \approx \frac{1}{2\pi\beta(1-\beta)n}{\left(\frac{1}{\beta^{2\beta}(1-\beta)^{2(1-\beta)}}\right)}^n\]
    and \[
    \frac{{(\Delta n-\Delta k)!}^2}{(\Delta n)!(\Delta n-2\Delta k)!}
    \approx \frac{1-\beta}{\sqrt{1-2\beta}}{\left(\frac{(1-\beta)^{2\Delta (1-\beta)}}{(1-2\beta)^{\Delta (1-2\beta)}}\right)}^n.\]
Hence, if \[
    \frac{(1-\beta)^{2\Delta(1-\beta)}}{\beta^{2\beta}(1-\beta)^{2(1-\beta)}(1-2\beta)^{\Delta(1-2\beta)}}
    <1,  \]
then $\mathbb{P}(W) \rightarrow 0$ as $n \rightarrow \infty$. Thus, with positive probability $G'$ is a $\Delta$-regular graph with no bi-hole of size at least $k = \beta n$ (for $n$ sufficiently large), as stated.
\end{proof}

\noindent
Performing explicit computations in \Cref{lem-pairing} for specific values of~$\Delta$ yields the
following bounds (see also Table~\ref{table:small-delta}).

\begin{corollary}\label{cor-upperbounds}
There exists~$N_0$ such that if~$n\ge N_0$, then $f^*(n, 3)<.4591n$,
 $f^*(n,4)<.4212n $, $f^*(n,5)<.3887n$, and
 $f^*(n,6)<.3621n$.
\end{corollary}

In particular, one sees that $f(n, 3) \le f^*(n, 3) < .4591n$ for sufficiently large~$n$. Thus, combined
with our earlier work, it follows that $.3411n < f(n,
3) < .4591 n$. It would be very interesting to improve either the lower or upper bound.

\subsection{Bounding \texorpdfstring{$f(n, \Delta)$}{f(n,Delta)} when \texorpdfstring{$\Delta$}{Delta} is large}\label{sec:large-delta}

In this section we address the behavior of $f(n, \Delta)$ for large $\Delta$ and prove
\Cref{thm:large-random}. Before doing so, let us note the following simple result, which shows that
\Cref{thm:upper} is tight (up to constants) when $\Delta$ is linear in $n$.

\begin{proposition}\label{prop:linear-delta}
For any $\varepsilon \in (0, 1)$ there is a constant $c = c(\varepsilon)$ such that for $n$ sufficiently
    large $f(n, (1 - \varepsilon)n) \ge c\log n$.
\end{proposition}
\begin{proof}
Let $\varepsilon$ be a positive constant, $\varepsilon <1$.
To show the lower bound on $f(n, (1 - \varepsilon)n)$,  consider a
bipartite graph~$G$ with parts~$A, B$ of size~$n$ each such that $\deg(x) \le (1 - \varepsilon)n$
for every~$x \in A$.
Letting~$G'$ be the bipartite complement of~$G$, we see that~$G'$ has at least
$\varepsilon n^2$ edges. The result then follows from the fact that for any $\varepsilon \in (0, 1)$ and
sufficiently large~$n$, any $n \times n$ bipartite graph with~$\varepsilon n^2$ edges contains a~$K_{t,
t}$ where $t = c \log n$ for some constant $c = c(\varepsilon)$. This can be proved using the standard
K\H{o}v{\'a}ri-S{\'o}s-Tur{\'a}n~\cite{KST54} double counting argument.
\end{proof}

Therefore, the behavior of~$f(n, \Delta)$ is clear whenever~$\Delta$ is linear, aside from more precise
estimates of the constants involved. What happens when~$\Delta$ is very large,
more precisely, when $\Delta = n - o(n)$? This is partly addressed in \Cref{thm:large-random}, which we now prove.

\begin{proof}[Proof of Theorem~\ref{thm:large-random}]

Recall that the classical Zarankiewicz number,~$z(n, t)$, is the largest number of edges in an $n\times
n$ bipartite graph that contains no copy of~$K_{t,t}$.  Assume first that~$t\geq 4$.

\bigskip
The lower bound on~$\Delta(t)$, 
follows from standard bounds on Zarankiewicz numbers.  Indeed, $z(n, t) \le
    Cn^{2 - 1/t}$ for some constant $C = C(t)$,  see for example \cite{KST54}.
    Thus, any $n\times n$ bipartite graph on at least $Cn^{2 - 1/t}$ edges
    contains a copy of $K_{t,t}$, and so any  $n\times n$ bipartite graph on at
    most $n^2-  Cn^{2 - 1/t}$ edges contains a bi-hole of size $t$.  In
    particular, any $n\times n$ bipartite graph with maximum degree at most $n
    - Cn^{1-1/t}$ contains a bi-hole of size $t$.  So the announced lower bound
    in \Cref{thm:large-random} holds.

\bigskip

To determine the stated upper bound on~$\Delta(t)$, we shall prove the
existence of a~$K_{t, t}$-free bipartite $n\times n$ graph with the additional
    constraint that the minimum degree of vertices (on one side) is large.  For
    that, we shall alter the standard random construction used to prove lower
    bounds on Zarankiewicz numbers.  For a graph~$F$, we shall carefully
    control~$X=X_F$,  the total number of copies of~$K_{t, t}$ in~$F$, as well
    as $X(v)=X_F(v)$, the number of copies of~$K_{t,t}$ containing a vertex~$v$ in~$F$.

\bigskip
    
 Let~$N \coloneqq 2n$ and $p \coloneqq cN^{-2/(t+1)}$, for a constant~$c$ to be determined later.  
Consider the bipartite  binomial random graph $G' \sim G(N, n, p)$ with parts~$A$ and~$B$ of sizes~$N$ and~$n$, respectively.
By Markov's inequality, $\mP(X \ge 2\mE[X]) \le 1/2$.
Since $\deg(v)$, for~$v\in A$, is distributed as~$\bin(n, p)$, Chernoff's inequality~\eqref{cher-low} from \Cref{lem:chernoff} 
with $\varepsilon \coloneqq 1/2$, implies that  with high probability, every vertex $v \in A$ has degree at least~$pn/2$. 
So with positive probability we  have $X \le 2\mE[X] = 2 \binom{N}{t}\binom{n}{t} p^{t^2} \le 2\binom{N}{t}^2p^{t^2}$ and $\deg(v) \ge pn/2$ for every~$v \in A$.

\bigskip
    
    Fix a bipartite graph~$G$ with these
    properties, i.e., $G$ is a bipartite graph with parts $A$ and $B$,  the number $X=X_G$ of copies of~$K_{t,t}$ satisfies $X\leq 2 \binom{N}{t}^2 p^{t^2}$ 
    and $\deg(v) \ge pn/2$ for every~$v \in A$.   Observe that there are fewer than~$n$ vertices~$v$ in~$A$ with
    $X(v) > \frac{2t}{n} 2 \binom{N}{t}^2 p^{t^2}$. Indeed, otherwise 
   $X\geq  n \frac{4t}{n} \binom{N}{t}^2 p^{t^2} /t$, a contradiction.

   \bigskip
   
   Let~$A' \subset A$ be a set of~$n$ vertices such that $X(v) \le \frac{2t}{n} 2 \binom{N}{t}^2 p^{t^2}$ for all $v \in
A'$.  Let~$H'$ be the subgraph of~$G$ induced by~$A' \cup B$. Finally, let $H$ be obtained from $H'$ by removing
    an edge from each copy of $K_{t, t}$. Thus~$H$ has no copies of~$K_{t,t}$.
It remains to check that the degrees of vertices in~$A$ are sufficiently large.
Indeed, for any~$v\in A$, 
\begin{align*}
\deg_H(v) &\geq  \deg_{H'}(v) - X(v) \\
 & \geq  np/2 - \frac{4t}{n}  \binom{N}{t}^2 p^{t^2} \\
 &\geq  Np/4 - \frac{4t}{n} N^{2t}p^{t^2} \\
 & \geq  \frac{c}{4}N^{1 -
    \frac{2}{t+1}}- \frac{4t}{n}  N^{2t}p^{t^2} \\
    &> (c/4 - 4tc^{t^2})n^{1 - \frac{2}{t+1}}\\
    & \geq \frac{1}{16}n^{1 - \frac{2}{t+1}},
    \end{align*}
where the last inequality holds for~$c=1/2$.
This concludes the proof of the general upper bound for~$t\geq 4$.

\bigskip

Now, let~$t=2$ or~$t=3$.    
We have $z(n, 2) = (1 + o(1))n^{3/2}$ and $z(n, 3) = (1 + o(1))n^{5/3}$, see K\H{o}vári, Sós \&
Turán~\cite{KST54}, Füredi~\cite{Fur96}, and   Alon, Rónyai \& Szabó~\cite{ARS99} (we refer the
interested reader to the survey by Füredi \& Simonovits~\cite{FS}).  In fact, the extremal constructions
are almost regular and therefore show that there are $n\times n$ bipartite graphs with no~$K_{2,2}$
(no~$K_{3,3}$) with minimum degree at least $(1+o(1))n^{1/2}$ (at least $(1+o(1))n^{2/3}$), respectively.
This completes the proof.
\end{proof}

\section{Concluding remarks}\label{conclusions}

We have made progress in determining the asymptotic behavior of~$f(n, \Delta)$. However, we could not
obtain better bounds for small~$\Delta$. The most glaring open problem is the case~$\Delta = 3$.

\begin{problem}\label{prob:delta-equals-3}
Determine the value of~$f(n, 3)$ for~$n$ sufficiently large.
\end{problem}

Further, recall that \Cref{thm:lower} does not cover the full range of $\Delta$ that is covered in
\Cref{thm:upper}. In particular, \Cref{thm:lower} shows that $f(n, \Delta) =
\Omega(\frac{\log\Delta}{\Delta}n)$ for $n \ge C\Delta \log \Delta$, while the upper bound in
\Cref{thm:upper} holds for $n \ge C' \frac{\Delta}{\log \Delta}$. It would be interesting to address what
happens in the range between these two estimates. 

\begin{problem}\label{prob:tricky-range}
Determine whether or not $f(n, \Delta) = \Theta(\frac{\log \Delta}{\Delta}n)$ whenever
    $c\frac{\Delta}{\log\Delta} \le n \le  c'\Delta\log\Delta$, for suitable positive constants $c$ and~$c'$.
\end{problem}

\section*{Acknowledgements} 
The authors thank David Conlon for interesting discussions and motivating them to further investigate this problem.

\end{document}